\documentclass[a4paper]{amsart}
\usepackage{amssymb,latexsym}
\usepackage[utf8x]{inputenc}
\usepackage{amsmath,amsthm}
\usepackage{amsfonts,mathrsfs}

\usepackage{todonotes}

\usepackage{enumerate,units}
\usepackage[all]{xy}
\usepackage{graphicx}

\usepackage{xcolor,url}
\usepackage{hyperref}
\hypersetup{
    colorlinks,
    citecolor=blue,
    filecolor=blue,
    linkcolor=blue,
    urlcolor=blue
}
\usepackage{cleveref}

\usepackage{tikz-cd}

\theoremstyle{plain}
\newtheorem{theorem}{Theorem}[section]
\newtheorem{observation}[theorem]{Observation}
\newtheorem{corollary}[theorem]{Corollary}
\newtheorem{lemma}[theorem]{Lemma}
\newtheorem{proposition}[theorem]{Proposition}

\newtheorem{fact}[theorem]{Fact}
\newtheorem*{claim}{Claim}
\newtheorem*{theorem*}{Theorem}

\theoremstyle{definition}
\newtheorem{definition}[theorem]{Definition}

\theoremstyle{remark}
\newtheorem{remark}[theorem]{Remark}

\numberwithin{equation}{section}
\newcommand{\forkindep}[1][]{%
  \mathrel{
    \mathop{
      \vcenter{
        \hbox{\oalign{\noalign{\kern-.3ex}\hfil$\vert$\hfil\cr
              \noalign{\kern-.7ex}
              $\smile$\cr\noalign{\kern-.3ex}}}
      }
    }\displaylimits_{#1}
  }
}

\newcommand\R{\mathbb{R}}		
\newcommand\C{\mathbb{C}}
\newcommand\Q{\mathbb{Q}}
\newcommand\N{\mathbb{N}}
\newcommand\Z{\mathbb{Z}}
\newcommand\U{\mathcal{U}}
\newcommand\F{\mathbb{F}}
\newcommand\HH{\mathbb{H}}

\newcommand\OO{\mathcal{O}}
\newcommand\PP{\mathcal{P}}

  \newcommand{\set}[1]{\left\{ {#1} \right\}}
\newcommand{\vect}[1]{\langle {#1} \rangle}
\newcommand{\abs}[1]{\lvert {#1} \rvert}

  \newcommand{\res}{\mathrm{res}}  
   \newcommand{\An}{\mathrm{An}}

\newcommand{\alg}{\mathrm{alg}}
\newcommand{\Gal}{\mathrm{Gal}}
\newcommand{\chara}{\mathrm{char}}
\newcommand{\ac}{\mathrm{ac}}

\begin{document}
\begin{abstract}
Milliet asks the following question: given two prime numbers $p\neq q$, is there a division algebra of characteristic $p$ which is of dp-rank $q^2$ and of dimension $q^2$ over its center? We answer in the affirmative. We also give an example of a finite burden central division algebra over some ultraproduct of $p$-adic numbers. As a conclusion we revisit an example of Albert to prove that there exists 
non-cyclic division algebras of finite dp-rank.
\end{abstract}

\title{Cyclic and non-cyclic division algebras of finite dp-rank}

\author[C. d\textquoteright Elb\'ee]{Christian d\textquoteright Elb\'ee$^\dagger$}
\thanks{$^\dagger$ Supported by ISF grant No. 1254/18.}

\address{$^\dagger$Einstein Institute of Mathematics\\
	The Hebrew University of Jerusalem\\
	Givat Ram 9190401, Jerusalem\\
	Israel} 
	\email{christian.delbee@mail.huji.ac.il} 
	\urladdr{http://choum.net/\textasciitilde chris/page\textunderscore perso/}

\maketitle

\section{Introduction}

\emph{Question of Milliet:} for every prime number $p\neq q$ does there exist a division algebra of characteristic $p$ of dp-rank $q^2$ and of dimension $q^2$ over its center?\\

Let us draw a quick picture of what is known on dp-finite division algebras. First, an easy computation gives that any dp-finite division algebra is finite-dimensional over its center \cite{DEfinitedprankalg}. In characteristic $0$, the celebrated Hamilton's quaternions $\HH$ is a division algebra central over the reals and of dp-rank $4$. It is also well-known that there are many central finite dimensional division algebras over the $p$-adic fields, which gives other examples of dp-finite division algebras in characteristic $0$ (see for instance \cite[Chapter 17]{Pie82}). In positive characteristic, the only known examples are given by Milliet in \cite[Theorem 2.1]{Mil19}:
\begin{itemize}
    \item a division algebra of characteristic $p\neq 2$ and of dp-rank $4$, and of dimension $4$ over its center;
    \item a division algebra of characteristic $2$ and of dp-rank $9$ and dimension $9$ over its center.
\end{itemize}
For this Milliet uses Dickson's construction of cyclic algebras (\cite[Chapter 5]{Lam91}), which he applies with a dp-minimal center. This paper uses essentially the same strategy as in the proof of  \cite[Theorem 2.1]{Mil19} to show that his conjecture holds. The existence of a cyclic division algebra relies mainly on a problem called by algebraic number theorists the ``solvability of the norm equation", which is simply the following question: given a Galois field extension $K/F$, is the norm application $N_{K/F}: K\rightarrow F$ surjective? We give a description of the norm application of a cyclic extension of a valued field (Theorem~\ref{thm_resnorm}), a purely algebraic result which we believe to be new. Besides giving a positive answer to Milliet's question, Theorem~\ref{thm_resnorm} has an extra consequence: there exists cyclic division algebras of finite burden over some ultraproduct of $p$-adic fields (Corollary \ref{cor_uqp}). \\

 Let $D$ be a finite dimensional division algebra central over a dp-minimal field $F$. By Johnson's classification of dp-minimal fields \cite{johnson-dp-minimal}, $F$ is either algebraically closed, real-closed or admits a definable Henselian valuation. It is standard that $F$ cannot be finite \cite[(13.11) Theorem]{Lam91} nor algebraically closed \cite[Ex. 13.10]{Lam91}. It is also standard that $F$ is real-closed if and only if $D$ is elementary equivalent to the quaternion algebra $\HH$ over $\R$ \cite[p. 209]{Lam91}. So we may focus on the case where $D$ is central over a field with a definable Henselian valuation, $(F,v)$. Those are almost\footnote{See the discussion after Theorem 1.2 in \cite{johnson2015arxiv}.} all elementary equivalent to fields of Hahn series $K((t^\Gamma))$ where $K$ is $\F_p^\alg$ or a local field of characteristic $0$ ($\R$, $\C$ or a finite extensions of $\Q_p$) and the group $\Gamma$ is dp-minimal. It is known that there are no nontrivial division algebra of finite dimension central over $\C((t^\Z))$, as the latter is \emph{quasi-algebraically closed} (see \cite[19.2]{Pie82} and \cite[Theorem 12]{Lan52}). Also, all finite dimensional division $\Q_p$-algebras are cyclic \cite[Chapter 17]{Pie82}, and so are all known dp-finite division algebras. Thus one may ask the following:
\begin{center}
    Are all dp-finite division algebra cyclic?
\end{center}
The answer is no, we give an example in Section \ref{sec_noncyclic} of a non-cyclic division algebra of dimension $16$ over the dp-minimal field $\R((X))((Y))$, based on a classical example of Albert \cite{Alb32}.\\

\noindent \textbf{Notations and conventions.} For a field $F$, we denote $F^\times$ the multiplicative group of $F$ and $F^{\times n}$ the subgroup of $F^\times$ of $n$-th powers. For a given valued field $(K,v)$, we use the notation $vK$ for the value group, $Kv$ for the residue field and $\OO_v$ for the valuation ring. We consider that the residue map $\res : \OO_v \rightarrow Kv$ extends to $K$ by sending $K\setminus \OO_v$ to zero.

\section{Cyclic algebra and the main result}

\noindent \textbf{Cyclic algebras.} Let $K/F$ be a cyclic field extension of degree $n$. Let $X$ be an indeterminate, and consider the formal $K$-vector space
\[D = K\cdot 1 \oplus K\cdot X \oplus \ldots \oplus K\cdot X^{n-1}.\]
Let $\Gal(K/F)$ be the Galois group of $K$ over $F$, and $\sigma_0$ a generator of $\Gal(K/F)$. We define on $D$ a multiplication in the following way. Let $\alpha\in F\setminus \set{0}$, we set: \[ X^n = \alpha \quad \quad X\cdot b = \sigma_0(b) X \ \ (\forall b\in K)\]
and extend to $D$ according to the distributive law. This turns $D$ into an algebra over $F$ of dimension $n^2$, of center $F$ \cite[(14.6)]{Lam91} ($D$ is \emph{central over $F$}), we denote it $D=(K/F,\sigma_0, \alpha)$, the \emph{cyclic algebra associated with $K/F$, $\sigma_0$ and $\alpha$}. For example, the quaternions $\HH$ is the cyclic algebra $(\C/\R,\sigma,-1)$ for $\sigma$ the complex conjugation.

The norm application $N = N_{K/F}: K \rightarrow F$ is defined as $N(a)= \prod_{\sigma\in \Gal(K/F)} \sigma(a) = \prod_{k = 0}^{q-1} \sigma_0^k(a)$. The following is \cite[(14.8)]{Lam91}.
\begin{fact}\label{fact_divalg}
Suppose that $n$ is prime, then the cyclic algebra $(K/F,\sigma_0, \alpha)$ is a division algebra if and only if $\alpha\notin N(K)$.
\end{fact}

\begin{observation}\label{obs_1}
Let $D$ be an algebra central over a field $F$ and of finite dimension $n$ over $F$. Then $D$ is definable in the cartesian power $F^{n}$. In particular if $F$ is dp-minimal, then $D = (K/F,\sigma_0,\alpha)$ is of dp-rank $n^2$.
\end{observation}
\begin{proof}
It is obvious that we may define the vector space structure of $D$ in $\bar D = F^{n}$. We show that we may define a multiplication in $\bar D$ such that $D$ and $\bar D$ are isomorphic as $F$-algebras. 
Let $c_1,\ldots, c_{n}$ be a fixed basis of $D$ as an $F$-vector space. Let $\lambda_k^{i,j}\in F$ be such that $c_i c_j = \sum_k \lambda_k^{i,j} c_k$. The map $\phi:D\times D \rightarrow D$ defined by $(d,d')\rightarrow dd'$ is a bilinear map over $F$.
For each $k$, let $M_k$ be the square matrix of size $n\times n$ given by $(\lambda_k^{i,j})_{i,j}$. The projection of $\phi$ to the space $Fc_k$ is a bilinear form whose matrix in the basis $c_1,\ldots,c_{n}$ is given by $M_k$. Thus, the multiplication of two given row vectors $a = (a_1,\ldots a_{n})$ and $b = (b_1,\ldots, b_{n})$ is given by the row vector
\[(aM_1 ^tb, \ldots, aM_{n}^tb).\]
Where $^tb$ is the transpose of $b$. Hence the multiplication is clearly definable from $F$ with parameters $(\lambda_k^{i,j})_{1\leq i,j,k\leq n}$. If $F$ is dp-minimal, the dp-rank of $D$ is $n$ by additivity of the dp-rank \cite{KOU13}.
\end{proof}

\begin{theorem}\label{theo_main}
Let $k_0$ be a dp-minimal field of characteristic $p\geq 0$ and $q$ a prime number different from $p$. Then there exists a dp-minimal field extension $F$ of $k_0$ and a cyclic division algebra $D$ of dp-rank $q^2$ whose center is $F$ and of dimension $q^2$ over $F$.
\end{theorem}

\begin{proof}[Proof of Theorem~\ref{theo_main}, Part 1]
If $p>0$, let $\Gamma$ be the subgroup of $\Q$ spanned by $\set{\frac{1}{p^n}\mid n\in \N}$, if $p=0$, let $\Gamma = \Z$. Then $\Gamma$ is $p$-divisible (if $p>0$) and not $q$-divisible, in the sense that $q\Gamma, q\Gamma +1, \ldots, q\Gamma + (q-1)$ are distinct cosets. As $\Gamma/r\Gamma$ is finite for all prime $r$, $\Gamma$ is dp-minimal as an ordered abelian group, by \cite[Proposition 5.1]{JSW17}.

Let $k= k_0((x^\Gamma))$ and $F = k((t^\Gamma))$. Let $v = v_t$ be the valuation associated to the variable $t$ in $F$. Using \cite[Theorem 9.8.1]{johnson}, the valued field $(F,v)$ is dp-minimal provided the following three conditions hold:
\begin{enumerate}
    \item the residu field and the value group are dp-minimals;
    \item $v$ is henselian and defectless;
    \item the value group is $p$-divisible (if $p>0$).
\end{enumerate}
 (3) is clear by choice of $\Gamma$. For (2), it is standard that any field of Hahn series is Henselian and defectless \cite{Khu01}. For (1), as $\Gamma$ is dp-minimal, we may appply \cite[Theorem 9.8.1]{johnson} to $(k, v_x)$ to get dp-minimality of the latter, since the residue field of $(k, v_x)$ is $k_0$, which is dp-minimal by hypothesis. Then, as $k$ is the residue field of the valuation $v$, using again \cite[Theorem 9.8.1]{johnson}, $F$ is dp-minimal.

Let $u$ be such that $u^q = t$, and let $K = F(u)$. By \cite{KSW11}, $k_0$ contains every root of unity, so $F(u)$ is the splitting field of the polynomial $X^q - t$. By classical Galois theory (for instance \cite[Proposition 3.7.6]{Wei06}) $K$ is a cyclic extension of $F$ of order $q$. Let $\Gal(K/F)$ be the Galois group of $K$ over $F$, and $\sigma_0$ a generator of $\Gal(K/F)$.

Now consider $D = (K/F, \sigma_0, x)$, $D$ is an $F$-algebra of dimension $q^2$. Using Observation~\ref{obs_1}, $D$ is an $F$-algebra of dp-rank $q^2$. It remains to show that $D$ is a division algebra, which, by Fact~\ref{fact_divalg} is equivalent to show that $x\notin N(K)$ for $N = N_{K/F}$ the norm application. This is the main step of the proof and follows from the following Theorem.
\end{proof}

\begin{theorem}\label{thm_resnorm}
Let $(F,v)$ be a valued field with value group $\Gamma$, of equicharacteristic. Let $q$ be a prime number and assume that $t\in F$ is such that $q\Gamma, q\Gamma +v(t) ,\ldots, q\Gamma + (q-1)v(t)$ are distinct cosets in $\Gamma$. Assume that $F$ contains all primitive $q$-th roots of $1$. In particular, $K = F(u)$ where $u^q = t$ is a proper cyclic extension of $F$ of degree $q$. 

Let $Fv$ be the residue field of $F$, $\OO_v$ the valuation ring of $v$ and $\res: \OO_v \rightarrow Fv$ the residue map, extended to $F$. Then \[\res(N_{K/F}(K))\setminus \set 0 = Fv^{\times q}.\]
\end{theorem}
The proof of Theorem~\ref{thm_resnorm} is given in Section~\ref{sec_proof_thmresnorm}.

\begin{proof}[Proof of Theorem~\ref{theo_main}, Part 2]
Recall that $F = [k_0((x^\Gamma))]((t^\Gamma))$, and that $K = F(u)$ for $u^q=t$. We prove that $x \notin N(K)$. As $\Gamma = \vect{\frac{1}{p^n}\mid n\in \N}$, the cosets $q\Gamma,q\Gamma + 1,\ldots,q\Gamma + (q-1)$ are distinct. We can apply Theorem~\ref{thm_resnorm}, to get that $\res(N_{K/F}(K))\setminus \set{0} = [k_0((x^\Gamma))]^{\times q}$. Again, as $q$ is not invertible in $\Gamma$, it is clear that $x\notin [k_0((x^\Gamma))]^{\times q}$.
\end{proof}


For a prime $q$ let $\PP_q$ be the set of prime numbers $p$ such that $(p-1)$ is divisible by $q$. By the theorem of Dirichlet on prime numbers, $\PP_q$ is infinite. It is easy to see that $p\in \PP_q$ if and only if $\F_p$ contains all $q$-th roots of $1$ and $\F_p^{\times q}\subsetneq \F_p^\times$ (the set of $q$-th roots of $1$ is the kernel of the map $x\mapsto x^q$).

\begin{corollary}\label{cor_uqp}
Let $q$ be a prime number, let $\U$ be a non-principal ultrafilter on the set $\PP_q$, and let $F = \prod_\U \Q_p $. Then there exists a division algebra over $F$ of degree $q^2$, and of burden between $q^2$ and $2^{q^2}$.
\end{corollary}
\begin{proof}
We fix a prime number $q$. It is clear that $F$ is a Henselian valued field with value group $\Gamma\equiv (\Z,+,0,<)$ and residue field $k=\prod_\U \F_p$. $F$ is inp-minimal by \cite[Theorem 4]{CS19}. It is standard that for each $p$, the only roots of unity contained in $\Q_p$ are the $(p-1)$-th roots. For $p\in \PP_q$, $q$ divides $(p-1)$, thus $\Q_p$ contains all $q$-th roots of unity, and so does $F$.

As $\Gamma \equiv \Z$, there exists $t\in F$ such that $q\Gamma, \ldots, q\Gamma +(q-1)v(t)$ are distinct cosets. By Theorem~\ref{thm_resnorm}, considering $K = F(\sqrt[q]{t})$, we have that $K/F$ is cyclic of order $q$ and $\res(N_{K/F}(K))\setminus \set{0} = k^{\times q}$.

For all $p\in \PP_q$, $\F_p^{\times q}\subsetneq \F_p^{\times}$. As $k= \prod_\U \F_p$, it follows that $k^{\times q}\subsetneq k^\times$. Let $x\in F\setminus \res^{-1}(k^{\times q})$ and let $\sigma_0$ be a generator of $\Gal(K/F)$. By Fact~\ref{fact_divalg}, the $F$ algebra $D = (K/F, \sigma_0, x)$ is a division algebra, central over $F$. By Observation \ref{obs_1}, $D$ is definable from $F$ in $F^{q^2}$. $D$ is an $F$-algebra of dimension $q^2$, hence by submultiplicativity of the burden \cite{Cher14}, the burden of $D$ is between $q^2$ and $2^{q^2}$.
\end{proof}
\begin{remark}
Chernikov \cite{Cher14} conjectures that the burden is subadditive. If the conjecture holds, then the division algebra in Corollary~\ref{cor_uqp} would be of burden $q^2$.
\end{remark}

\begin{remark}
It is easy to deduce from Theorem~\ref{thm_resnorm} (using the same method as in the proof of Corollary \ref{cor_uqp}) that for two given primes $q$ and $p\in \PP_q$, as $\F_p^{\times q}\subsetneq \F_p^\times$, there exists a central division algebra  of dimension $q^2$ over $\F_p((X))$. Then, using Ax-Kochen-Ershov ($\prod_\U \F_p((X))\equiv \prod_\U \Q_p$) we recover Corollary \ref{cor_uqp}. 
\end{remark}

\section{The norm of a cyclic extension of a valued field}\label{sec_proof_thmresnorm}

This section is dedicated to the proof of Theorem~\ref{thm_resnorm}. We start by proving the following result.
\begin{proposition}\label{prop_norm}
Let $K = F(u)$ be a cyclic field extension of a field $F$ of prime degree $q\neq \chara (F)$, with $u^q = t\in F$. Assume that $F$ contains all primitive $q$-th root of unity. Then, there exists a set $C_0$ of tuples $(k_0,\ldots, k_{q-1})\in \set{0,\ldots ,q-1}^q$ such that $\sum_i k_i \in q\Z$, and there exists a function $f : C_0 \rightarrow \Z\setminus \set{0}$ such that for any $a = b_0+b_1 u+\ldots +b_{q-1}u^{q-1}$, we have
\[N_{K/F}(a) = \sum_{(k_0,\ldots, k_{q-1})\in C_0} f(k_0,\ldots,k_{q-1}) b_{k_0} \ldots b_{k_{q-1}} t^{\frac{\sum_i k_i}{q}}.\]
\end{proposition}

The proof of Proposition~\ref{prop_norm} is a bit lengthy and tedious, we start by some preliminary results and notations. We fix a cyclic field extension $K/F$ of degree $q$, there exists $u\in K$ and $t\in F$ such that $K = F(u)$ and $t = u^q$. As $\Gal(K/F)$ is cyclic, there is a generator $\sigma_0$ and there is a primitive $q$-th root of unity $\xi\in F$ such that $\sigma_0(u) = \xi u$.
Let $a\in K$, we write $a = b_0+b_1u+\ldots+b_{q-1}u^{q-1}\in K$, $b_i\in F$, as $1,u,u^2,\ldots,u^{q-1}$ is a basis of $K$ as an $F$-vector space. Then 
\[\sigma_0^k(a) = \sum_{i = 0}^{q-1} \xi^{ik}b_iu^i\]
It follows that 
\[N(a) = \prod_{k=0}^{q-1} \sum_{i = 0}^{q-1} \xi^{ik}b_iu^i.\]
Be developping, we obtain
\begin{align}
    N(a) &= \sum_{(i_0,\ldots,i_{q-1})\in \set{0,\ldots,q-1}^{q}} b_{i_0}u^{i_0}\cdot b_{i_1}u^{i_1}\xi^{i_1}\cdot \ldots \cdot b_{i_{q-1}}u^{i_{q-1}}\xi^{i_{q-1}(q-1)}\\
    &= \sum_{(i_0,\ldots,i_{q-1})\in \set{0,\ldots,q-1}^{q}} (b_{i_0}\ldots b_{i_{q-1}}) u^{i_0+\ldots+i_{q-1}}\xi^{0.i_0+1.i_1+\ldots+(q-1)i_{q-1}}\label{norm0}
\end{align}
The arithmetic in the powers of $\xi$ is modulo $q$, hence we will study the map $(i_0,\ldots, i_{q-1})\mapsto 0.i_0+1.i_1+\ldots +(q-1)i_{q-1}$ in the finite field $\F_q$, and identify $\set{0,\ldots,q-1}$ with $\F_q$. We will also identify the indices $0,\ldots, q-1$ of a tuple $(c_0,\ldots, c_{q-1})\in \F_q^q$ as elements of $\F_q$. 

Let $\sim $ be the equivalence relation on $\F_q^q$ defined as follows: for $c = (c_0,\ldots,c_{q-1}),d = (d_0,\ldots,d_{q-1})\in \F_q^q$  
\begin{align*}
    c\sim d &\iff \text{$c$ is an anagram of $d$}\\
    &\iff \text{there is a permutation $\sigma\in \mathfrak{S}(\F_q)$ of the indices}\\
    &\quad \quad \quad\text{ such that $(c_{\sigma(0)},\ldots,c_{\sigma(q-1)}) = (d_0,\ldots,d_{q-1}$})
\end{align*}

Let $\tilde \Sigma 
: \F_q^q\rightarrow \F_q$ be the linear form defined by $\tilde \Sigma(c_0,\ldots,c_{q-1}) = \sum_{i=0}^{q-1} ic_{i}$. For $c\in \F_q^q$, we denote by $\An(c)$ the set of anagrams of $c$, i.e. the class of $c$ modulo $\sim$, and for each $\lambda\in \F_q$, we set
\[\An^\lambda(c) = \set{d\in \An(c)\mid \tilde\Sigma(d) = \lambda}\]
Every class $\An(c)$ contains an element of the form 
\[(\underbrace{c_1,\ldots,c_1}_{l_1\text{ times}}, \underbrace{c_2,\ldots,c_2}_{l_2\text{ times}},\ldots,\underbrace{c_k,\ldots,c_k}_{l_k\text{ times}})\]
where $1 \leq k\leq q$ and $l_1+\ldots+l_k = q$. It is standard that $\abs{\An(c)} = \frac{q!}{l_1! \cdot \ldots \cdot l_k !}$.

\begin{lemma}\label{lm_combin}
Let $c = (c_0,\ldots,c_{q-1})\in \F_q^q$. 
\begin{enumerate}
    \item For all $\lambda,\mu\in \F_q^\times$, $\abs{\An^\lambda(c)} = \abs{\An^\mu(c)}$.
    \item If there exists $i\neq j$ such that $c_i\neq c_j$, then $\abs{\An^0(c)} = \abs{\An^1(c)}$ if and only if $c_0+\ldots+c_{q-1} \neq 0$.
\end{enumerate}
\end{lemma}
\begin{proof}
\textit{(1)} (Action of $(\F_q^\times, \cdot)$ on $\An(c)$). Let $\nu\in \F_q^\times$, the map $x\mapsto \nu x$ defines a permutation of $\F_q$ (of inverse $x\mapsto \nu^{-1} x$) we denote it $\sigma_\nu\in \mathfrak S (\F_q)$. This permutation induces a permutation $f_\nu$ of $\F_q^q$ by acting on the indices in the following way: if $d=(d_0,\ldots,d_{q-1})\in \F_q^q$ then $f_\nu(d) = d^{\sigma_\nu} = (d_{\sigma_\nu(0)},\ldots, d_{\sigma_\nu(q-1)})$. Furthermore, if $d\in \An(c)$ then $f_\nu(d)\in \An(c)$, hence $\An(c)\subset \F_q^q$ is stable by $f_\nu$. As $f_\nu$ is injective, $f_\nu: \An(c)\rightarrow \An(c)$ is bijective. Let $d\in \An(c)$. Then \[\tilde{\Sigma}( d^{\sigma_\nu}) = \sum_{i=0}^{q-1}i d_{\sigma_\nu(i)} = \sum_{j=0}^{q-1} \sigma_{\nu}^{-1}(j)d_j = \sum_{j = 0}^{q-1}\nu^{-1}j d_j = \nu^{-1} \tilde{\Sigma}(d).\]
It follows that for each $\nu\in \F_q^\times$, $f_\nu : \An(c)\rightarrow \An(c)$ sends $\An^\lambda(c)$ to $\An^{\nu^{-1}\lambda}(c)$. In particular $\An^0(c)$ is stable by $f_\nu$ and for all $\lambda,\mu\in \F_q^\times$, $\abs{\An^\lambda(c)} = \abs{\An^\mu(c)}$.\\
\textit{(2)} (Action of $(\F_q,+)$ on $\An(c)$). Let $\sigma\in \mathfrak S(\F_q)$ be the permutation defined by the map $x\mapsto x+1$. This permutation induces a permutation $f_\sigma$ of $\F_q^q$ by acting on the indices in the following way: for $d = (d_0,\ldots,d_{q-1})\in \An(c)$, $f_\sigma (d) = d^\sigma = (d_{\sigma(0)},\ldots, d_{\sigma(q-1)}) = (d_1,\ldots,d_{q-1},d_0)$. It is clear that $\An(c)$ is stable by $f_\sigma$, hence $f_\sigma: \An(c)\rightarrow \An(c)$ is bijective. Let $d\in \An(c)$, and  $\nu = c_0+\ldots+c_{q-1}$. then 
\[\tilde\Sigma(d^\sigma) - \tilde\Sigma(d) = \sum_{i=0}^{q-1}id_{i+1} - \sum_{i=0}^{q-1}id_i = \sum_{j=0}^{q-1}(j-1)d_j - \sum_{i=0}^{q-1}id_i = -\sum_{i=0}^{q-1}d_i = -\nu.\]
It follows that $f_\sigma$ maps $\An^0(c)$ to $\An^{-\nu}(c)$. Similarly, for all $0\leq j \leq q-1$, the iterate $f_\sigma^j = f_\sigma\circ \ldots \circ f_\sigma$ maps $\An^0(c)$ to $\An^{-j\nu}(c)$. 

Assume that $\nu=c_0+\ldots+c_{q-1}\neq 0$. As the additive group $(\F_q,+)$ is cyclic, it is generated by any non-zero element, hence as $f_\sigma$ and all its iterate are injective, we have $\abs{\An^0(c)} = \abs{\An^\lambda(c)}$ for all $\lambda\in \F_q$. 

Assume now that $c_0+\ldots+c_{q-1} = 0$ and that $\abs{\An^0(c)} = \abs{\An^\lambda(c)}$. As $c_0+\ldots +c_{q-1} = 0$, $\tilde{\Sigma}(d^\sigma) = \tilde{\Sigma}(d)$ for all $d\in \An(c)$, hence $\An^\lambda(c)$ is stable by $f_\sigma$ for all $\lambda\in \F_q$. Observe that the group $(\F_q,+)$ acts on $\An^0(c)$ by the map $(k,d)\mapsto f_\sigma^k(d)$. As there exists $i\neq j$ such that $c_i\neq c_j$, for any $d\in \An^0(c)$, the stabiliser of $d$ under this action is $\set{0}$. It follows that $q$ divides $\abs{\An^0(c)}$. As $\An(c) = \bigsqcup_{i=0}^{q-1}\An^i(c)$, $\abs{\An(c)} = q \abs{\An^0(c)}$. As $q$ divides $\abs{\An^0(c)}$, $q^2$ divides $\abs{\An(c)} = \frac{q!}{l_1!\ldots l_k !}$, a contradiction.
\end{proof}

\begin{remark}
If $c_0+\ldots+ c_{q-1} = 0$, then $q(q-1)$ divides $\abs{\An^0(c)}$, since both $(\F_q,+)$ and $(\F_q^\times, \cdot)$ act on $\An^0(c)$. If $c_0+\ldots + c_{q-1}\neq 0$, then for all $\lambda\in \F_q$ we have $\abs{\An^\lambda(c)} = \frac{(q-1)!}{l_1!\ldots l_k!}$.
\end{remark}

We can now give the proof of Proposition~\ref{prop_norm}. 
\begin{proof}[Proof of Proposition~\ref{prop_norm}]
For $c, d\in \F_q^q$, if $c\sim d$, then $b_{c_0}\cdot \ldots \cdot b_{c_{q-1}} = b_{d_{0}}\cdot \ldots \cdot b_{d_{q-1}}$. We denote $b_{d_{0}}\cdot \ldots \cdot b_{d_{q-1}}$ by $\bar b_d$. Let $C$ be a set of representatives of $\F_q^q/\sim$. For $c\in \F_q^q$, let $\Sigma(c)$ be the \emph{integer} $c_0+\ldots+c_{q-1}$ for representatives $c_i\in \set{0,\ldots,q-1}$. From equation \ref{norm0} we deduce the following:
\begin{align}
    N(a) &= \sum_{c\in C} \bar b_c \abs{An(c)} u^{\Sigma(c)} \left(\sum_{d\in \An(c)} \xi^{\tilde\Sigma(d)}\right)\\
    &= \sum_{c\in C} \bar b_c \abs{An(c)} u^{\Sigma(c)}\left(\sum_{i=0}^{q-1} \abs{\An^i(c)}\xi^i\right)\label{norm1}
\end{align}
There are three cases depending on $c\in C$:
\begin{itemize}
    \item If $c_i = c_j$ for all $i\neq j$ then $\abs{\An(c)} = 1$, $\tilde\Sigma(c) = 0$ and $\Sigma(c) = qs$ for $s$ a representative of $c_0$ in $\Z$.
    \item If $c_i\neq c_j$ for some $i\neq j$, and $c_0+\ldots+c_{q-1} \neq 0$, then by Lemma~\ref{lm_combin}, $\sum_{i=0}^{q-1} \abs{\An^i(c)}\xi^i = \abs{\An^0(c)}(\sum_{i=0}^{q-1}\xi^i) = 0$.
    \item If $c_i\neq c_j$ for some $i\neq j$, and $c_0+\ldots+c_{q-1} = 0$, then by Lemma~\ref{lm_combin}, $\sum_{i=0}^{q-1} \abs{\An^i(c)}\xi^i = \abs{\An^0(c)} - \abs{\An^1(c)} \neq 0$. Further, $\Sigma(c) = qs$ for some integer $1\leq s\leq q-1$.
\end{itemize}
It follows that only tuples $c\in C$ such that $\Sigma(c)\in q\Z$ appear as summand in equation \ref{norm1}. Let $C_0$ be the set of tuples $c\in C$ such that $c_0+\ldots+c_{q-1} = 0$, equivalently, $\Sigma(c)\in q\Z$. Let $f(c) = \abs{\An(c)}(\abs{\An^0(c)}-\abs{\An^1(c)})\in \Z$. For every $c\in C_0$, we have that $f(c)\neq 0$.
We get:
\begin{align}
    N(a) &= \sum_{c\in C_0}  f(c)\bar b_c u^{\Sigma(c)}\\
    N(a) &= \sum_{c\in C_0}  f(c)\bar b_c t^{\frac{\Sigma(c)}{q}}\label{norm2}
\end{align}
By identifying back $\F_q$ with $\set{0,\ldots,q-1}$ we get the description of the norm in Proposition~\ref{prop_norm}.
\end{proof}

\begin{proposition}\label{prop_normval}
Let $(F,v)$ be a valued field with value group $\Gamma$, of equicharacteristic. Let $q$ be a prime number and assume that $t\in F$ is such that $q\Gamma, q\Gamma +v(t) ,\ldots, q\Gamma + (q-1)v(t)$ are distinct cosets in $\Gamma$. Assume that $F$ contains all primitive $q$-th roots of $1$. In particular, $K = F(u)$ where $u^q = t$ is a proper cyclic extension of $F$ of degree $q$. Then, for all $a=b_0+b_1u+\ldots+b_{q-1}u^{q-1}\in K$, with $b_i\in F$, we have
\[v(N(a)) = \min\set{iv(t)+ qv(b_i)\mid i=0,\ldots, q-1}.\]

\end{proposition}
\begin{proof}
First, as $q\Gamma\cap q\Gamma + v(t) = \emptyset$ we have that $u\notin F$. Now since $q$ is prime and $F$ contains all primitive $q$-th roots of unity, it follows from classical Galois theory, that $K = F(u)$ is a cyclic extension of $F$ (see for instance \cite[Proposition 3.7.6]{Wei06}). From Proposition \ref{prop_norm}, we have that 
\[v(N(a)) = v[\sum_{\bar k=(k_0,\ldots, k_{q-1})\in C_0}  f(\bar k) b_{k_0} \ldots b_{k_{q-1}} t^{\frac{\sum_i k_i}{q}}].\] 
Observe that $iv(t)+qv(b_i) = v[f(\bar k) b_{k_0},\ldots, b_{k_{q-1}} t^{\frac{\sum_i k_i}{q}}]$ for $\bar k = (i,\ldots, i)\in C_0$. We have that for all $i\neq j$, $q\Gamma+iv(t)$ and $q \Gamma + jv(t)$ are disjoint, hence for all $b_i,b_j\in F$, $iv(t)+qv(b_i) \neq  jv(t)+q v(b_j)$. We deduce that $v(N(a)) = \min\set{iv(t)+ qv(b_i)\mid i=0,\ldots, q-1}$ follows from the following claim.
\begin{claim}
For all $\bar l\in C_0$, such that there exists $i\neq j$ with $l_i\neq l_j$, there exists $\bar k=(i_0,\ldots,i_0)\in C_0$ such that 
\[v(f(\bar l) b_{l_1}\ldots b_{l_{q-1}} t^{\frac{\sum_i l_i}{q}}) > v(f(\bar k) b_{k_0}\ldots b_{k_{q-1}}  t^{\frac{\sum_i k_i }{q}}).\]
\end{claim}
Let $\bar l = (l_0,\ldots , l_{q-1})$ and for each $i = 0,\ldots , q-1$, $\bar k^i = (l_i,\ldots , l_i)$. Note that $f(\bar k)$ belongs to the prime field hence $v(f(\bar k)) = 0$ since $(F,v)$ is equicharacteristic. We have 
\begin{align*}
    v(f(\bar l) b_{l_0} \ldots b_{l_{q-1}} t^{\frac{\sum_i l_i}{q}}) &= v(b_{l_0} \ldots b_{l_{q-1}} t^{\frac{\sum_i l_i}{q}})\\
    &= \sum_{i=0}^{q-1} v(b_{l_i}) + \frac{\sum_i l_i}{q}\\
    &= \frac{1}{q} \left(\sum_{i=0}^{q-1} qv(b_{l_i})\right) +\frac{\sum_i l_i}{q}\\
    &= \frac{1}{q} \left(\sum_{i=0}^{q-1} qv(b_{l_i}) + l_i \right)\\
    &= \frac{1}{q} \left(\sum_{i=0}^{q-1} v( f(\bar k ^i) b_{i}\ldots b_i  t^{\frac{\sum_j k^i_j }{q}})\right)
\end{align*}
It follows that $v(f(\bar l) b_{l_0} \ldots b_{l_{q-1}} t^{\frac{\sum_i l_i}{q}})$ is the arithmetic mean of the set \[\set{v( f(\bar k^i) b_{k_0^i}\ldots b_{k_{q-1}^i} t^{\frac{\sum_j k^i_j }{q}}) \mid i = 0,\ldots, q-1}\]
As there are $i\neq j$ such that $l_i\neq l_j$ there exists $i_0$ such that $v(f(\bar l) b_{l_1} \ldots b_{l_{q-1}} t^{\frac{\sum_i l_i}{q}}) > v(f(\bar k^{i_0}) b_{k^{i_0}_0}\ldots b_{k^{i_0}_{q-1}}  t^{\frac{\sum_i k_i^{i_0}}{q}})$.
\end{proof}

\begin{proof}[Proof of Theorem \ref{thm_resnorm}]
Let $x\in \OO_v$ such that $\res(x)\neq 0$ and assume that there exists $a= b_0+ b_1 u +\ldots + b_{q-1} u^{q-1}\in K$ with $b_i\in F$ such that $N(a) = x$. Then, using Proposition~\ref{prop_normval}, as $v(x) = 0$, we have 
\[v(N(a)) = \min\set{iv(t)+ qv(b_i)\mid i=0,\ldots, q-1} = 0.\]
As $q\Gamma, q\Gamma + v(t), \ldots ,q\Gamma + (q-1)v(t)$ are distinct, we have 
$q v(b_0) = 0$ and $iv(t)+qv(b_i)>0$ forall $i= 1,\ldots,q-1$. Observe that  $b_0^q = f(\bar k) b_{k_0}\ldots b_{k_{q-1}} t^{\frac{\sum_i k_i}{q}}$, hence by Proposition~\ref{prop_norm}, we have that $x =N(a) = b_0^q + y$ where $y\in F$ and $v(y) >0$. It follows that $v(x- b_0^q) >0$ hence $\res(x) = \res(b_0^q) \in Fv^{\times q}$. We show that $\res(N(K)) \subseteq Fv^{\times q}$. Let $\res(x)\in Fv^{\times q}$. Then there exists $y\in F$ such that $\res(x) = \res(y)^q = \res(y^q)$. Let $z = x- y^q$, then, since $y+z\in F$, $N(y+z) = (y+z)^q = y^q +z g(y,z)$ for some polynomial $g$. As $v(z)>0$, we have that $\res(x) = \res(y^q) = \res(N(y+z))$ hence $\res(x)\in \res(N(K))$.
\end{proof}

\section{A non-cyclic division algebra of finite dp-rank}\label{sec_noncyclic}

Let $R = \R((X))$ the field of Laurent series over the reals and $F = R((Y)) = \R((X))((Y))$ the field of Laurent series over $R$. By \cite[Theorem 9.8.1]{johnson}, $F$ is dp-minimal. By making $X$ infinitesimal in $R$ and $Y$ infinitesimal in $F$, $F$ is an ordered field with $0<Y<X$.

\begin{lemma}\label{lm_anisotropic}
Let $K$ be a quadratic extension of $F$ the form $K = F(\gamma)$ for $\gamma = \sqrt{\alpha^2+\beta^2}$, with $\alpha,\beta\in F$. If $a_1,a_2,a_3,a_4,a_5,a_6\in K$ are such that 
\[Xa_1^2-a_2^2+Xa_3^2 = -Xa_4^2+Ya_5^2+XYa_6^2\]
then $a_i = 0$ for all $i.$
\end{lemma}
\begin{proof}
Assume that we have such $a_1,\ldots, a_6$. We also have 
\begin{equation}
    X(a_1^2+a_3^2+a_4^2-Ya_6^2) = a_2^2+Ya_5^2 \label{eq1}
\end{equation}
As $K = F\oplus \gamma F$, write $a_i = b_i+\gamma c_i$, we have $a_i^2 = b_i^2+(\alpha^2+\beta^2)c_i^2 + \gamma b_ic_i$. As $(1,\gamma)$ is an $F$-independent tuple, \ref{eq1} implies 
\begin{equation}
    X(a_1'^2+a_3'^2+a_4'^2-Ya_6'^2) = a_2'^2+Ya_5'^2 
\end{equation}
for $a_i'^2 = b_i^2+(\alpha^2+\beta^2)c_i^2$. It follows that in \ref{eq1}, we may assume that $a_i^2$ are sum of squares of elements in $F$.

Let $\ac_Y : F\rightarrow R$ be the angular component map, i.e. $\ac_Y$ maps $\sum_{i} r_i Y_i$ to $r_{i_0}$, where $i_0$ is the smallest $i$ such that $r_i\neq 0$. Similarly we have another angular component map $\ac_X : R\rightarrow \R$. 

In $a_1^2+a_3^2+a_4^2-Ya_6^2$, the nonzero monomial of smallest exponent in $Y$ is either of the form $f(X)Y^{2n}$ for some $n\in \Z$ and $f(X)$ a sum of squares in $R$, or of the form $-g(X)Y^{2n+1}$ for some $n\in \Z$ and $g(X)$ a sum of squares in $R$.

In both cases, $\ac_Y(X(a_1^2+a_3^2+a_4^2-Ya_6^2)) = \pm Xh(X)$ for $h(X)$ a sum of squares in $R$. Thus, the nonzero monomial in $\pm Xh(X)$ of smallest exponent in $X$ is of the form $ rX^{2m+1}$, for some $m\in \Z$ and $r\in \R$. 

It is clear that $\ac_Y(a_2^2+Ya_5^2)$ is a sum of squares in $R$, so the nonzero monomial in $\ac_Y(a_2^2+Ya_5^2)$ of smallest exponent in $X$ is of the form $r'X^{2k}$ for $r'\in \R$ $r'>0$ and $k\in \Z$. 

Applying $\ac_Y$ on both sides of \ref{eq1}, we have that $rX^{2m+1} = r'X^{2k}$ which leads to $r = r' = 0$. As $r = \ac_X(\ac_Y(X(a_1^2+a_3^2+a_4^2-Ya_6^2)))$ and $r' = \ac_X(\ac_Y(a_2^2+Ya_5^2))$, we have that both sides of the equation \ref{eq1} equals zero. On the right hand side of \ref{eq1}, $a_2^2$ and $Ya_5^2$ are positive elements in $F$ hence $a_2 = a_5 = 0$. In the left hand side of \ref{eq1}, we get \begin{equation}
    a_1^2+a_3^2+a_4^2=Ya_6^2. \label{eq2}
\end{equation}
Similarly as before, in the left hand side of \ref{eq2}, the nonzero monomial in $Y$ of smallest coefficient is of the form $f(X)Y^{2n}$. In the right hand side of \ref{eq2}, the nonzero monomial in $Y$ of smallest coefficient is of the form $g(X)Y^{2n+1}$. Thus, both sides of equation \ref{eq2} are zero, hence in particular $a_6 = 0$. Finally, as $a_1^2,a_3^2,a_4^2$ are positive elements in $F$, we have that $a_1 = a_3 = a_4 = 0$.
\end{proof}
Observe that $F$ actually admits such quadratic extensions. For instance, take $\alpha = \frac{1}{1-X}$ and $\beta = \frac{1}{1+X}$, then it is easy to check that $\alpha^2+\beta^2 = \sum_{n\geq 0} (2n+1)X^{2n}$ does not have a square root in $F$.

\begin{definition}[Quaternion algebra]
Let $F$ be a field of characteristic not $2$. Let $u,v\in F^\times$, we define the \emph{quaternion algebra} $\left(\frac{u,v}{F}\right)$ to be the set of expressions of the form $a + b i + c j + d ij$ for $a,b,c,d\in F$ and symbols $i,j$. It is clearly an $F$-vector space of dimension $4$. We define a multiplication on $\left(\frac{u,v}{F}\right)$ based on the rules:
\[i^2 = u,\quad j^2 = v,\quad ij=-ji\]
with which $\left(\frac{u,v}{F}\right)$ is an $F$-algebra.
\end{definition}

It is classical that $\HH = \left(\frac{-1,-1}{\R}\right)$ and that every quaternion algebra is cyclic.

\begin{theorem}\label{thm_alb}
For $F = \R((X))((Y))$, we define $D_1 = \left(\frac{X,-1}{F}\right)$ and $D_2 = \left(\frac{-X,Y}{F}\right)$. Let $D$ be the biquaternion algebra $D_1\otimes_F D_2$. Then $D$ is a non-cyclic division algebra of dimension $16$ over $F$ and of dp-rank $16$.
\end{theorem}
\begin{proof}
We follows the classical proof as exposed in \cite[15.7]{Pie82} or \cite[2.10]{Jac96}.
We will use the following facts.
\begin{fact}[Theorem 2.10.3 in \cite{Jac96}]\label{fact_albert}
Let $K$ be a field of characteristic not $2$, and $u,v,u',v'\in K^\times$. Then the biquaternion algebra $\left(\frac{u,v}{K}\right)\otimes_K \left(\frac{u',v'}{K}\right)$ is a division algebra if and only if the quadratic form $\phi : K^6 \rightarrow K$ defined by \[\phi(a_1,a_2,a_3,a_4,a_5,a_6) = ua_1^2+va_2^2-uva_3^2 - u'a_4^2 - v'a_5^2 +u'v'a_6^2\]
is anisotropic, i.e. $\phi(\bar a) = 0$ only if $\bar a = 0$.
\end{fact}
\begin{fact}[Corollary 13.4 in \cite{Pie82}]\label{fact_tensorfield}
Let $F$ be a field and $D$ a division algebra of dimension $n^2$ over $F$. Let $K$ be a finite extension of $F$ of degree a prime divisor of $n$. Then $K$ is isomorphic to a subfield of $D$ if and only if $D\otimes_F K$ is not a division algebra.
\end{fact}
\begin{fact}[Lemma 2.10.2 in \cite{Jac96}]\label{fact_quartic}
Let $F$ be a field not containing $\sqrt{-1}$ and let $L$ be a cyclic quartic extension field of $F$. Then the (unique) quadratic extension of $F$ lying in $E$ has the form $F(\sqrt{\alpha^2+\beta^2})$ where $\alpha, \beta\in F$ and $\alpha^2+\beta^2$ is not the square of an element of $F$.
\end{fact}
\begin{claim}
For all quadratic field extension $K$ of $F$ of the form $K= F(\sqrt{\alpha^2+\beta^2})$, the algebra $D\otimes_F K$ is a division algebra.
\end{claim}
\begin{proof}[Proof of the claim]
It is standard that $D\otimes_F K = (D_1\otimes_F K)\otimes_K (D_2\otimes_F K)$, $D_1\otimes_F K = \left(\frac{X,-1}{K}\right)$ and $D_2\otimes_F K = \left(\frac{-X,Y}{K}\right)$ (see for instance \cite[Corollary 9.4.a, Corollary 9.4b]{Pie82}). The claim follows from Fact~\ref{fact_albert} and Lemma~\ref{lm_anisotropic}.
\end{proof}
In particular $D$ is a division algebra. By the claim and Fact \ref{fact_tensorfield}, we have that all quadratic extension of $F$ of the form $F(\sqrt{\alpha^2+\beta^2})$ are not isomorphic to a subfield of $D$. Assume that $D$ is cyclic, hence there exist a cyclic extension $L$ of $F$ of degree $4$ such that $D = (L/F, \sigma_0, \gamma)$, for some $\gamma\in F$ and $\sigma_0$ a generator of $\Gal(L/F)$. By Fact \ref{fact_quartic}, there exists a quadratic extension of $F$ lying in $L$ of the form $F(\sqrt{\alpha^2+\beta^2})$, a contradiction. It follows from Observation~\ref{obs_1} that the division algebra $D$ from Theorem~\ref{thm_alb} is of dp-rank $16$.
\end{proof}

\bibliographystyle{alpha}
\bibliography{cyclicdpfinite}
\end{document}